\documentclass[letterpaper, 10 pt, conference]{ieeeconf}

\IEEEoverridecommandlockouts                              

\usepackage{empheq}
\usepackage{amsmath,amssymb,amsfonts}
\usepackage{environ}
\usepackage{cite}
\usepackage{grffile}
\usepackage{multicol}
\usepackage{graphics}
\usepackage{subcaption}
\usepackage{url}
\usepackage{float}
\newcommand{\mcl}[1]{\mathcal{ #1}}
\newcommand{\mbf}[1]{\mathbf{ #1}}
\newcommand{\norm}[1]{\left\Vert #1\right\Vert}

\newcommand{\ip}[2]{\left\langle{#1},{#2}\right\rangle}

\newcommand{\bmat}[1]{\begin{bmatrix} #1\end{bmatrix}}
\newcommand{\R}{\mathbb{R}}

\newtheorem{thm}{Theorem}
\newtheorem{defn}[thm]{Definition}
\newtheorem{lem}[thm]{Lemma}

\newtheorem{cor}[thm]{Corollary}

\newcommand{\PI}{\pmb{\Pi}}
\newcommand{\pie}{\scalebox{0.9}[1.2]{$\mathit{\Pi}$}}
\newcommand{\fourpi}[4]{\hspace{0.5mm}\pie\hspace{-1.mm}\left[\begin{array}{c|c}
		#1&#2\\\hline #3 & \{#4\}
	\end{array}\right]}
\allowdisplaybreaks

\title{\LARGE \bf
	Static Output Feedback Stabilization of Linear Systems with Multiple Delays}

\author{Danilo Braghini$^{1}$, Eduardo S. Tognetti$^{2}$ and Matthew M. Peet$^{3}$
	\thanks{This work was supported by the National Science Foundation under grants No. 2337751 and 2429973, and in part by Conselho Nacional de Desenvolvimento Cient\'{\i}fico e Tecnol\'{o}gico (CNPq), Brazil, under grant No. 302987/2025-8.}
	\thanks{$^{1}$Danilo Braghini\{{\tt\small dbraghini@asu.edu}\} and $^{3}$Matthew M. Peet\{{\tt\small mpeet@asu.edu}\} are with the School for Engineering of Matter, Transport and Energy at Arizona State University, Tempe, AZ, USA.}
	\thanks{$^{2}$Eduardo S. Tognetti\{{\tt\small estognetti@ene.unb.br}\} is with the Electrical Engineering Department, at University of Brasília, Brasília, DF, Brazil.
	}%
}

\begin{document}
	\maketitle
	\thispagestyle{empty}
	\pagestyle{empty}
	\begin{abstract}                
		This work proposes a new procedure for the stabilization of time-delay systems using Static Output Feedback (SOF) control. A previous convex optimization approach to SOF for Ordinary Differential Equations (ODEs) is extended to time-delay systems through the use of a proposed state-space representation. This approach is based on solving two convex optimization problems, which are extensions of Linear Matrix Inequalities (LMIs) to infinite-dimensional systems. The first problem is stabilization under state feedback control; the second problem takes advantage of the Projection Lemma, which is extended here from matrices to Partial Integral (PI) operators. Finally, the results are compared with other SOF solutions for systems with delay found in the literature, showing a significant reduction in conservatism.
	\end{abstract}
	
	\section{Introduction}
    
	Many practical problems in control include models with delay. These delays may arise prima facie from representation of the physics (e.g., milling and cutting~\cite{Hajdu2020} and mine ventilation~\cite{sandou2008receding}) or from communication delay in the feedback loop. 
    Feedback control of systems with delay is particularly difficult in the case of output feedback, where only partial measurements of the system state are available. While methods for estimator design and implementation exist for time-delay systems, these approaches require real-time solution of an ODE coupled with a transport equation~\cite{das_2019CDC}. By contrast, static output feedback controllers may be implemented efficiently using static gains applied to the measured output.
    Despite the simplicity and reduced implementation costs, even SOF control of linear ordinary differential equations (ODEs) does not have a known convex formulation and is generally believed to be intractable~\cite{toker1995np}. 
Some efforts have been made to develop 
LMI-based conditions for SOF of linear systems with a single delay. For example, in~\cite{liu2015static} and~\cite{parlakci2024robust}, we find approaches based on the bounding of Lyapunov-Krasovskii functionals, where 
the resulting bilinearity is resolved by iteration. 
For positive systems with a measurement delay, 
a convenient restriction on the Lyapunov variables was considered in~\cite{huynh2019static}, which -- together with a change of variables -- convexifies the inequalities for a class of systems that satisfy a rank constraint. Nevertheless, the change of variables requires a generalized inverse to retrieve a non-unique controller. An innovative approach to the problem was taken in~\cite{Fridman2009}, wherein sufficient LMI conditions were proposed for existence of a sliding-mode controller for uncertain systems with time-varying delay and sufficiently small nonlinearities. In none of these cases, however, is the fundamental non-convexity of the SOF problem directly addressed. By contrast, in this paper, we focus first on identifying an efficient LMI approach to SOF for ODEs with demonstrated practical performance, and then generalize that approach 
using a recently developed state-space representation of infinite-dimensional systems.
        
  Specifically, for linear ODEs, an LMI-based SOF synthesis approach 
   was presented in~\cite{peaucelle2001efficient} using the Projection Lemma. 
   The key idea is to require closed-loop stability under full
   state feedback and SOF using a common quadratic Lyapunov function. Combined with the Projection Lemma, this approach allows one to compute the state feedback and the SOF controllers sequentially: 
   a two-step approach where
an LMI is solved in each step. Generalization of this approach to time-delay systems 
requires both efficient methods for state-feedback control, 
and a generalization of the projection lemma. As a result, application of the two-step methodology to time-delay systems has met with limited success -- See, e.g.~\cite{hao2015static}, where it is necessary to resort to an iterative method to get feasible solutions, even for single-delay systems.  
     
Recently, however, significant progress has been made in both state-space representation of time-delay systems and the synthesis 
of full state-feedback controllers for such systems using the Partial Integral Equation (PIE) framework. 
     PIEs can be used to represent systems with multiple discrete and distributed delays, neutral-type systems, and even PDEs with delays~\cite{peet2021representation,peet2021partial,shivakumar2024extension,danilo,jagt2024h}. Furthermore, the algebraic structure of the Partial Integral (PI) operators that parametrize PIEs allows for both a generalization of the LMI for optimal full-state feedback synthesis~\cite{shivakumar2022dual} and, as will be shown in Sec.~\ref{sec:sol}, the Projection Lemma. 

The goal of this paper, then, is to consider the problem of SOF stabilization of linear systems with multiple state and output delays. To achieve this, we first reformulate the problem as SOF of a PIE. Then, we synthesize a stabilizing state-feedback controller, $\mcl K$, using the methodology described in~\cite{shivakumar2022dual}. Then, we generalize the Projection Lemma to Linear PI Operator Inequalities (LPIs). Next, we formulate a bilinear PI operator inequality for SOF control of a PIE. Then, we augment the operator inequality with a second constraint that the Lyapunov variable used in the SOF inequality is also valid in the LPI used to obtain the given state-feedback controller, $\mcl K$. Finally, by applying the Projection Lemma, we show that these two operator inequalities (one of which is bilinear) are implied by a new single operator inequality, which is linear and can be solved by convex optimization. This approach is validated by solving this new LPI to obtain controllers for several time-delay systems.
    
	\textbf{Notation}: $L_2^m[a,b]$ is the space of Lebesgue square-integrable $\R^m$-valued functions on spatial domain $s \in 
    [a,b]$. 
    $W^{m}[a,b]$ is the Sobolev space of $L_2$-differentiable $\R^{m}$--valued functions on $s \in 
    [a,b]$. 
    We denote $\R L_2^{m,n}[a,b]:=\R^m \times L_2^{n}[a,b]$ and $\R W^{m,n}[a,b]:=\R^m \times W^{n}[a,b]$; both the spatial domain and dimensions may be omitted when clear from context. For Hilbert spaces $X,Y$, $\mcl L(X,Y)$ denotes the set of bounded linear operators from $X$ to $Y$.  We use the calligraphic font (e.g. $\mcl{A}$) to represent such bounded linear operators. For any $\mcl A\in \mcl L(Y,X)$, $\mcl A^*$ denotes the adjoint operator and for self-adjoint operators, $\mcl A \succcurlyeq 0$ means $\ip{\mbf x}{\mcl A \mbf x}\ge 0$ for all $\mbf x\in X$.  
	\section{Problem Formulation}\label{sec:introproblem}
	Consider the Delay Differential Equation (DDE):
	\begin{align}\label{DDE}
		\dot x(t) &= A x(t) + \sum_{i=1}^K A_i x(t-\tau_i) + B u(t), \,t\geq0,\\
		y(t) &=C x(t)+\sum_{i=1}^K C_{i} x(t-\tau_i), \notag
	\end{align}
	where $x(t) \in \R^m$, $u(t) \in \R^{n_u}$ is the control input, $y(t)~\in~\R^{n_y}$ is the measured output, $x(t)=x_0(t),$ for all $-\tau_K \leq t \leq 0$ where $x_0 \in W^{m}[- \tau_K, 0]$, and, for convenience, $\tau_{1 }< \tau_{2}< \cdots < \tau_{K}$. $A, A_i, B, C$, and $C_i$ are matrices of appropriate dimensions.
    
	In the SOF problem, we want to find a static gain matrix $L~\in~\R^{n_u \times n_y}$ such that if $u(t)=Ly(t)$, the closed-loop system.\vspace{-3 mm}
    
\begin{align}\label{DDE_CL}
\hspace{-2 mm}	\dot x(t)& = \left( A + B L C \right) x(t)+\sum_{i=1}^K\left( A_i + B L C_i \right) x(t-\tau_i),
\end{align}
 is exponentially stable.
\section{Review: SOF control of ODEs}\label{sec:ODE}
When no delays are present in the model (i.e., $A_i=C_i=0$), the two-step approach of~\cite{peaucelle2001efficient} may be used to find an SOF controller. In this section, we briefly review this method. 

First, we note that in the ODE case, an SOF gain, $L$, stabilizes Eqn.~\eqref{DDE} if and only if there exists a positive definite matrix $P \succ 0$ such that \[\left(A + B L C \right)^T P+  P \left(A +  B L C \right) \prec 0.\] However, this matrix inequality is bilinear in the variables $L$ and $P$. Now, 
suppose that we are given a stabilizing state-feedback controller, $u(t)=Kx(t)$. Then we know there exists some $P\succ0$ such that \[ \left(A + B K \right)^T P+  P \left(A +  B K \right)\prec~0.\]
    
    Next, recall the Projection Lemma of~\cite{pipeleers2009extended}.
 \begin{lem}\label{eliminationODE}\cite{pipeleers2009extended}
     Given a symmetric matrix $Q$ and two
matrices $U$ and $V$ of column dimension $m$; there exists an unstructured matrix $F$ that satisfies
\begin{align}\label{projecionlem1}
     Q+U^{T}F^{T}V+V^{T}FU \prec 0,
\end{align}
if and only if the following projection inequalities with respect to $F$ are satisfied
		\begin{align}\label{projecionlem2}
	N_{u}^{T}Q N_{u}  \prec 0, \quad N_{v}^{T}Q N_{v}  \prec 0,
	\end{align}
    where $N_u$ and $N_v$ are arbitrary matrices whose columns form a basis of the null spaces of $U$ and $V$, respectively.
 \end{lem}
 
Now, we make the conservative assumption that if the system is SOF controllable, there exists some $L$ and $P\succ0$ such that $ \left(A + B L C \right)^T P+  P \left(A +  B L C \right) \prec 0$ AND $ \left(A + B K \right)^T P+  P \left(A +  B K \right)\prec~0$. By the Projection Lemma, one can show that this is equivalent to the existence of 
$L$, $F$, and $P\succ 0$ such that
 \begin{align*}
          \hspace{-3mm}  \bmat{-F-F^T
          &B^TP +FLC\\
  P B+C^TL^TF^T& P \left(A + B K \right)+ \left(A + B K \right)^TP} \prec 0.
\end{align*}

By defining the invertible change of variables $Z=FL$, this nonlinear matrix inequality is equivalent to the LMI
 \begin{align*}
          \hspace{-3mm}  \bmat{-F-F^T
          &B^TP +ZC\\
  P B+C^TZ^T& P \left(A + B K \right)+ \left(A + B K \right)^TP} \prec 0,
\end{align*}
where the resulting SOF gains are given by $L=F^{-1}Z$.
    \section{Review: PIE representation of DDEs and LPIs for Stability and Stabilization}\label{sec:intropie}
    In this section, the PIE representation is briefly introduced, which allows the formulation of an LPI, a generalization of an LMI, to solve the SOF problem for the delay system in Eqn.~\eqref{DDE}. The main result presented in Sec.~\ref{sec:sol} also requires us to define exponential stability and recall two prior results: an LPI for stability analysis and an LPI for full-state feedback stabilization, which are reproduced here for completeness.
    
	\subsection{4-PI Operators}
    
	 In the PIE representation 
     the systems are parametrized by the class of partial integral operators defined in the following. For a 
     comprehensive 
     overview, see Sec. II of~\cite{shivakumar2024extension}.
	
	\begin{defn}\label{def:4PI}
		Given a matrix $P$ and polynomials $Q_1,Q_2,R_0$, $R_1$, and $R_2$, a 4-PI operator $\mcl P=~\fourpi{P}{Q_1}{Q_2}{R_i} \in\mcl L(\R L_2^{m,n},\R L_2^{p,q})$ is such that
		{\small
			\begin{align*}
				&\left(\mcl P \bmat{x\\\mbf x} \right)(s) := \bmat{Px + \int_{-1}^{0}Q_1(\theta)\mbf x(\theta)d\theta\\Q_2(s)x+ \mcl R\mbf{x} (s)}, \text{where}\\
				&\left(\mcl R\mbf x\right)(s)\hspace{-1.5mm}~= ~\hspace{-1.5mm}R_0(s) \mbf x(s) +\hspace{-1.5mm}\int_{-1}^s  \hspace{-1.5mm}R_1(s,\theta)\mbf x(\theta)d \theta+\hspace{-1.5mm}\int_s^0 \hspace{-1.5mm}R_2(s,\theta)\mbf x(\theta)d \theta.
			\end{align*}
		}
		
		Furthermore, the set of 4-PI operators with dimensions $m$, $n$, $q$, $p$ is denoted by $\PI_{q,n}^{p,m}$, or $\PI_4$ when the dimensions are clear from context.
	\end{defn}
	
	If $p=m$ and $q=n$, the set of 4-PI operators is closed under composition, addition, and adjoint; explicit formulae for these operations can be obtained in terms of the polynomial matrices used to parameterize them~\cite{shivakumar2024extension}. Concatenation and inversion of PI operators are also defined in some cases; the reader may find precise definitions and formulae in~\cite{shivakumar2024extension} and~\cite{shivakumar2022dual}. 
    The associated dimensions ($m,n,p,q$) are inherited from the dimensions of the constant matrix $P\in \R^{p \times m}$ and polynomial matrices $Q_1(s) \in \R^{p \times n}$, $Q_2(s) \in \R^{q \times m}$, and $R_0(s),R_1(s,\theta), R_2(s,\theta) \in \R^{q \times n}$. In the case where a dimension is zero, we use $\emptyset$ in place of the associated parameter with zero dimension. 
%
%
		\subsection{PIE Representation of DDEs}
	In this subsection, we show how 4-PI operators can be used to parameterize DDEs. Specifically, we define a Partial Integral Equation (PIE) as a set of integro-differential equations of the form 
    	\begin{align}\label{PIE}
		\partial_t (\mcl T \mbf x(t)) &= \mcl A \mbf x(t)+\mcl B u(t), \\ 
        y(t)&= \mcl C \mbf x(t), \notag
	\end{align}
	where $\mcl T, \mcl A,\mcl B, \mcl C \in \PI_4$ are operators as in Def.~\ref{def:4PI}, and $\mbf x(t)\in~\R L_2^{m,mK}[-1,0]$.

 Now, for any $\{A,B,C,A_i,C_i\}$, define 
    		\begin{align}\label{Box:PIs}
			\mcl T &=\fourpi{I_m}{0}{\bmat{I_{m}&\cdots&I_{m}}^T}{0,0,-I_{mK}},\\ \mcl A &=\fourpi{A+\sum_{i=1}^K Ai}{-\bmat{A_1 \cdots A_K}}{0}{I_{\tau},0,0},\notag \\\mcl C &=\fourpi{C+\sum_{i=1}^K Ci}{-\bmat{C_1 \cdots C_K}}{\emptyset}
            {\emptyset},\notag\\
            \mcl B&=\fourpi{B}{\emptyset}{0}{\emptyset}, \, I_{\tau}=\text{diag}(I_m\tau_1^{-1},\ldots,I_m\tau_K^{-1}).
			\notag
		\end{align}
        
        Then the solutions of Eqn.~\eqref{DDE} and Eqn.~\eqref{PIE} are equivalent in the following sense~\cite{peet2021representation}.
	\begin{lem}\label{LemPIEDDE}
		For given $A,A_i,B,C,C_i,\tau_i$, suppose $\{\mcl{T,A,B,C}\}$ are as defined in Eqn.~\eqref{Box:PIs}. For any $\mbf x_{0} \in \R L_{2}^{m,mK}[-1,0]$ and $u \in L_{2}[0,\infty)$, if $x$, $y$ satisfies the DDE in Eqn.~\eqref{DDE} under input $u$ and initial condition $x_{0}$, then $\mbf x$ and $y$ satisfy Eqn.~\eqref{PIE}, where\vspace{-1 mm}
        \[\mbf x(t)=(x(t),\dot x(t+s\tau_1), \dots,\dot x(t+s \tau_K)),\] 
        for all $t \geq 0, s \in [-1,0]$, under input $u$ and initial conditions $x(0)=x_0(0)$ and $\dot x(s\tau_{i})$. 
        Likewise, if 
        \[\mbf x(t)=(x(t), \partial_s \phi_1(t,s),\dots, \partial_s \phi_K(t,s)),\]
        and $y$ satisfy Eqn.~\eqref{PIE}, under input $u$ and initial condition $ \mbf x(0)\hspace{-2 mm}~=~\hspace{-2 mm}(x(0),\partial_s \phi_1(0,s),\dots, \partial_s \phi_K(0,s)))$, then $x(t)$ satisfy Eqn.~\eqref{DDE} under input $u$ and initial condition $x_{0}(t)\hspace{-1 mm}~=~\hspace{-1 mm}\phi_K(0,t/\tau_K)$, for all $t \in [-\tau_{k},0]$, where $x_0(0)\hspace{-1 mm}~=~\hspace{-1 mm}x(0)$.
	\end{lem}
    
Lem.~\ref{LemPIEDDE} implies that there is a one-to-one map between solutions of Eqn.~\eqref{DDE} with static output feedback controller $u(t)=Ly(t)$ and solutions of the PIE defined as
	\begin{align}\label{PIE_CLo}
		\partial_t (\mcl T \mbf x(t))& = \left( \mcl A + \mcl B L \mcl C \right)\mbf x(t).
	\end{align}
In the following subsection, we define exponentially stability of a PIE, show that stability of a PIE implies stability of the corresponding DDE, and provide an LPI condition for stability and full-feedback stabilization of a PIE. 
    \subsection{LPIs for Exponential Stability and Stabilization}
	Having established equivalence of solutions between a multi-delay system and its associated PIE, let us now consider the questions of stability and stabilization. 
	\begin{defn}[Exponential Stability of a DDE]\label{def:stability}
		We~say~that the DDE in Eqn.~\eqref{DDE}, with $u=y=0$, is exponentially stable with decay rate $\alpha>0$ if there is a constant $M$ such that for any initial condition $x_0 \in \R W$, if $x(t)$ satisfies Eqn.~\eqref{DDE}, then $\norm{x(t)}_{\R^m} \leq M  \norm{x_0}_{L_{\infty}}e^{-\alpha t}$ for all $t\geq 0$, where $\norm{x_0}_{L_{\infty}} =\max\limits_{-\tau_K\leq t \leq 0} \norm{x_0(t)}_{\R^m}$.
	\end{defn}
	
	\begin{defn}[Exponential Stability of a PIE]\label{def:stabilityPIE} 
    
    We~say~that the PIE Eqn.~\eqref{PIE} is exponentially stable with decay rate $\alpha >0$ if there is a constant $M$ such that for any initial state $\mcl T \mbf x(0)$, with $u=y=0$, the solution $\mbf x(t)$ satisfies $\norm{\mcl T \mbf x(t)}_{\R L_{2}}\leq M  \norm{\mcl T \mbf x(0)}_{\R L_{2}} e^{-\alpha t}$, for all $t \geq 0$.
	\end{defn}
	
	Lem.~\ref{thmstability} shows that stability of the PIE is sufficient for stability of the delay system. 
	
	\begin{lem}{\label{thmstability}}
		For given $A,A_i,\tau_i$, suppose $\{\mcl{T,A}\}$ are as defined in Eqn.~\eqref{Box:PIs}. Then, if Eqn.~\eqref{PIE} is exponentially stable with decay rate $\alpha >0$, Eqn.~\eqref{DDE} is exponentially stable with decay rate $\alpha >0$. 
	\end{lem}
	\proof
	Suppose that $\{ \mcl{A,T}\}$ is exponentially stable as per Def.~\ref{def:stabilityPIE}. Let $x(t)$ be a solution of the DDE in Eqn.~\eqref{DDE}. Then, if we define $\mbf x(t)=(x(t),\dot x(t+s\tau_1), \dots,\dot x(t+s \tau_K))$, we have that $\mcl T \mbf x(t)=(x(t),x(t+s\tau_1), \dots,x(t+s \tau_K))$, and 
    as per Lem.~3, 
    $\mbf x$ satisfies Eqn.~\eqref{PIE}. Since  $\{ \mcl{A,T}\}$ is exponentially stable, this implies $\norm{\mcl T\mbf x(t)}_{\R L_2}  \leq M  \norm{\mcl T\mbf x(0)}_{\R L_2} e^{-\alpha t}$.
    
    It follows from the norm definitions that  $\norm{x(t)}_{\R^{m}} \leq \norm{\mcl T \mbf x(t)}_{\R L_{2}}$ for all $t\geq0$. Then, 
	\begin{align*}
		\norm{x(t)}_{\R^{m}}  \leq M  \norm{\mcl T \mbf x(0)}_{\R L_2} e^{-\alpha t},
	\end{align*}
	where $\norm{\mcl T \mbf x(0)}_{\R L_2} =\norm{x(0)}_{\R^m}+\sum_{i=0}^K\norm{x_{i0}}_{L_2}$, where $x_{i0}(s)=x(s\tau_i)$, for $s \in [-\tau_i,0]$. But, from Lem.~\ref{LemPIEDDE}, $\norm{x(0)}_{\R^m}=\norm{x_0(0)}_{\R^m} \leq\norm{x_0}_{L_{\infty}}$ and $\norm{x_{i0}}_{L_2} \leq \norm{x_{i0}}_{L_{\infty}} \leq \norm{x_0}_{L_{\infty}}$. Therefore,
    	\begin{align*}
		\norm{x(t)}_{\R^{m}}  \leq M (1+K) \norm{x_0}_{L_{\infty}} e^{-\alpha t}.
	\end{align*}
	\endproof
	
	The LPI for full-state feedback controller synthesis can be found in~\cite{shivakumar2022dual}. For completeness, we recall the result here.
	
	\begin{lem}[~\cite{shivakumar2022dual}]\label{lem:stability}
		Given $\mcl{A, T} \in \PI_4$, suppose there exist constants $\delta, \alpha> 0$ and $\mcl P= \mcl P^* \in \PI_4$ such that $\mcl P \succcurlyeq \delta I$ and
		\begin{align}\label{stabilityineq}
			\mcl A^* \mcl P& \mcl T+\mcl T^* \mcl P \mcl A \preccurlyeq -2\alpha \mcl T^* \mcl P \mcl T.
		\end{align}
		Then, the PIE Eqn.~\eqref{PIE}, with $u=y=0$, defined by $\{\mcl{A,T}\}$, is exponentially stable with decay rate $\alpha$.
	\end{lem}
	\proof
    The proof can be found in~\cite{shivakumar2022dual}.
	\endproof
%

\begin{lem}[~\cite{shivakumar2022dual}]\label{stage1}
	Given $\mcl{A,B,T} \in \PI_4$, suppose there exist $\mcl P= \mcl P^*$, $\mcl Z \in \PI_4$, and constants $\delta, \alpha > 0$ such that $\mcl P\succcurlyeq \delta I$ and
	\begin{align}{\label{LPIstage1}}
		\mcl A \mcl P \mcl T^* +\mcl T \mcl P \mcl A^* +\mcl B \mcl Z \mcl T^*+ \mcl T \mcl Z^* \mcl B^* &\preccurlyeq- 2 \alpha \mcl T \mcl P \mcl T^*.
	\end{align}
	Then, the PIE $\partial_t(\mcl T \mbf x(t))=\left(\mcl{A+BK} \right)\mbf x(t)$, 
    where $\mcl K = \mcl Z \mcl P^{-1}$, is exponentially stable with decay rate $\alpha$ \end{lem}
\proof
 The proof can be found in~\cite{shivakumar2022dual}. 
\endproof
\vspace{-2mm}

	\section{Main Results}\label{sec:sol}
After introducing a state-space representation of delay systems, showing the equivalence of exponential stability between the representations, and recalling an LPI for stability, we are ready to present the main results of this work. First, we partially extend the Projection Lemma from matrices to 4-PI operators, proving only sufficiency. Then, we use this lemma to derive a convex optimization solution for the SOF control of time-delay systems in two steps.
	\subsection{Projection Lemma for PIs}
	We propose a sufficient condition that extends Lem.~\ref{eliminationODE}, widely applied for analysis and control of ODEs, to PI operators. 
    However, we need to first define the right annihilator of a 4-PI operator.
	
	\begin{defn}\label{def:anihilator}
		Given $\mcl R \in \PI_{q,n}^{p,m}$
        , we say $\mcl S\in \PI_{n,l}^{m,k}$ 
        is a right annihilator of $\mcl R$ if $\mcl R (\mcl S \mbf x) = 0 \in \R L_2^{p,q}$, for all $\mbf x~\in~\R L_2^{k,l}$ and $\mcl R^{*}\mcl R \succcurlyeq \epsilon I$ for some constant $\epsilon > 0$.
	\end{defn}
        
	Then, sufficiency of the Projection Lemma, Lem.~\ref{eliminationODE}, can be extended to the algebra of 4-PI operators as follows. 
	
	\begin{lem}\label{lem:elimination}
		Consider $\mcl V \in \PI_{q,n}^{p,m}$, $\mcl U \in \PI_{s,n}^{r,m}$, and $\mcl Q =\mcl Q^* \in \PI_{n,n}^{m,m}$. Let $\mcl R$ and $\mcl S$ be right annihilators of $\mcl U$ and $\mcl V$, respectively. Then, if there exists $\mcl X \in \PI_4$, of appropriate dimensions, such that
		\begin{align}\label{proj2}
			\mcl{Q+U^*XV+V^*X^*U} \preccurlyeq 0,
		\end{align}
		then the two LPIs hold:
		\begin{align}\label{proj1}
			(\mcl S)^*\mcl Q \mcl S \preccurlyeq 0, \quad
			(\mcl R)^*\mcl Q \mcl R \preccurlyeq 0.
		\end{align}
	\end{lem}
	\proof
	By definition, Eqn.~\eqref{proj2} implies
	\begin{align*}
		\ip{\mbf x}{(\mcl{Q+U^*XV+V^*X^*U}) \mbf x} \leq 0,
	\end{align*}
	for all $\mbf x \in \R L_2^{m,n}$. Now, note that, for $\mbf x = \mcl S \mbf y$,\vspace{-1 mm}
	\begin{align*}
		\ip{\mcl S\mbf y}{(\mcl{Q+U^*XV+V^*X^*U})\mcl S \mbf y}\leq 0
	\end{align*}
	for some $\mbf y \in \R L_2^{k,l}$, implying
	\begin{align*}
		\hspace{-1mm}&\ip{\mbf y}{(\mcl S)^*\mcl Q\mcl S\mbf y}+\ip{\mbf y}{((\mcl S)^*\mcl{U^*X}(\mcl V\mcl S)+(\mcl V\mcl S)^*\mcl{X^*U}\mcl S)\mbf y} \leq 0.
	\end{align*}
	But $\mcl V\mcl S\mbf y =\mcl V(\mcl S\mbf y) = 0$, yielding $\ip{\mbf y}{(\mcl S)^*\mcl Q\mcl S\mbf y} \leq 0$,
	for all $\mbf y \in \R L_2^{k,l}$, resulting in the first inequality of Eqn.~\eqref{proj1}.
	
	Similarly, making $\mbf x = \mcl R \mbf v$ for some $\mbf v \in \R L_2^{i,j}$ yields the second inequality of Eqn.~\eqref{proj1}.
	\endproof \vspace{-2mm}

\subsection{Stabilizing Static Output Feedback Controller}\vspace{-1mm}

For a given 4-PI operator $\mcl K$ (obtained from the solution of the LPI in Lem.~\ref{stage1}), the following theorem provides an LPI 
for design of a stabilizing SOF controller. 
\begin{thm}\label{stage2}
	Given $\mcl{T, A, B, C}$, and $\mcl K \in \PI_4$, suppose there exist $\mcl P=\mcl P^*~\in~\PI_{n,n}^{m,m}$, $F~\in~ \R^{n_u \times n_u}$, $Z \in \R^{n_u \times n_y}$, and constants $\alpha,\delta, \epsilon > 0$, such that $\mcl P \succcurlyeq \delta I$ and $\Phi+\Phi^* \preccurlyeq 0$, where
		\begin{align}\label{eq:Phi}
			\Phi=\bmat{-F+\frac{\epsilon}{2} I
				& \mcl{B^* P T} + Z \mcl{C} -F\mcl{K}\\
				0& \mcl{T^*P(\mcl{A+BK+\alpha T})}}.
		\end{align}   
    Then, if $L = F^{-1}Z$, 
    the PIE defined by $\{\mcl T, \mcl A + \mcl B L \mcl C\}$
 is exponentially stable with decay rate $\alpha$.
\end{thm}
\proof
Our first goal is to show the invertibility of the matrix $F$. Note that $\Phi+\Phi^* \preccurlyeq 0$ implies $F+F^T\succcurlyeq \epsilon I \succ 0$. Then $
x^T(F+F^T)x > 0,
$ for all $x \in \R^{n_u}$. But $x^T(F+F^T)x=2x^TFx$, yielding $x^TFx > 0$. Next, from Cauchy-Schwartz inequality, $\norm{x}\norm{Fx} \geq|x^TFx| \geq x^TFx$. Consequently, $\norm{Fx} > 0$, implying that F is non-singular. 

Our task is now to show that $\Phi+\Phi^* \preccurlyeq 0$ implies exponential stability of 
the closed-loop PIE 
Eqn.~\eqref{PIE_CLo}. 
The key observation is that, substituting the matrix $Z=FL$, it is clear that \[
	\Phi+\Phi^*=\mcl Q + \mcl U^*F\mcl V + \mcl {V^*}F^*\mcl{U} ,
\]
where,
{\small
	\begin{align*}
		\mcl Q &:= \bmat{\epsilon I & \mcl{B^*PT}\\\mcl T^* \mcl P \mcl B & \mcl{T^* P}(\mcl{A+BK+\alpha T})+(\mcl{A+BK+\alpha T})^*\mcl{PT}},\\
		\mcl V &= \bmat{-I_{n_u} &L\mcl{C-K}}, 
        \\ \mcl U &= \bmat{I_{n_u}& 0}= \fourpi{\bmat{I_{n_u} & 0}}{0}{\emptyset}{\emptyset}. 
\end{align*}}

Now, consider the operator 
  \[  \mcl S=\bmat{L\mcl{C-K}\\ I_{mK}^{m}},\]
where $I_{mK}^{m}~\in~\PI_{mK,mK}^{m,m}$ is the identity 4-PI operator. Note that 
$
	\mcl V(\mcl S \mbf x) = -(L\mcl{C-K}) \mbf x + (L\mcl{C-K})\mbf x =0,  
$ and
$\mcl S$ is a right annihilator of $\mcl V$.
Thus, from Lem.~\ref{lem:elimination}, 
we have \[(\mcl S)^*\mcl Q \mcl S\preccurlyeq 0,\] implying
\begin{align*}
(\mcl{A^*+ C^*}L^T\mcl{B^*})\mcl{PT}+\mcl T^* \mcl P(\mcl{A+B}L\mcl{C}) \preccurlyeq -2\alpha\mcl{T^*PT},
\end{align*}
which proves stability of Eqn.~\eqref{PIE_CLo} using the 
Lyapunov functional $V(\mcl T \mbf x(t))=\ip{\mcl T\mbf x(t)}{\mcl P \mcl T \mbf x(t)}$,
by Lem.~\ref{lem:stability}.
\endproof 

\subsection{Application to Time-Delay Systems}
Although Thm.~\ref{stage2} provides conditions for closed-loop stability of a PIE, the resulting gain, $L$, may be directly applied to the associated PDE to ensure closed-loop stability.
\begin{cor}
For given $\{A,B,C,A_i,C_i\}$, suppose that $\mcl{T,A,B,C}$ are as defined in Eqn.~\eqref{Box:PIs} and $L,\alpha$ satisfies the conditions of Thm.~\ref{stage2}. Then for any $x$, $x_0$ which satisfies Eqn.~\eqref{DDE} with $u(t)=Ly(t)$, we have that $\norm{x(t)}_{\R^m}\le M \norm{x_0}_{L_{\infty}}e^{-\alpha t}$ for some $M>0$.
\end{cor}

\section{Numerical Examples}\label{sec:num}

In this section, we validate the proposed SOF synthesis condition by constructing the SOF gain, $L$, for several test cases and simulating the resulting closed-loop dynamics subject to non-zero initial conditions. The LPIs of Lem.~\ref{stage1} and Thm.~\ref{stage2} were computed using the Matlab toolbox PIETOOLS~\cite{manual2024} and the convex-optimization solver MOSEK, for the values of $\delta=10^{-6}$, and $\epsilon=10^{-4}$. 
Additionally, a time step of $0.01 s$ was used in PIESIM~\cite{peet2020piesim}, a numerical simulator integrated with PIETOOLS. 

\noindent \textbf{Example 1: }
 Consider the single-delay system from~\cite{TDScontrol}
 , wherein
\begin{align*} 
A&=\bmat{0 & 1 & 0 & 0\\  \frac{-2k}{M} & 0 & \frac{-mg}{M} & 0\\ 0 &0&0&1\\\frac{2k}{Ml}&0&\frac{(m+M)g}{Ml}&0},\,A_1=0, \, B=\bmat{0\\\frac{1}{M}\\0\\ \frac{-1}{Ml}},\\
C&=\bmat{C_0\\0}, \, C_1=\bmat{0\\C_0}, \,C_0=\bmat{1 & 0&0&0\\0 &0& 1&0}, \notag
\end{align*}
 $k = 1000 N/m$, $l = 0.4 m$, $g = 9.8 m/s^2$, $m = 0.1 kg$, $M = 1 kg$, and $\tau=0.1 s$. The PIE parametrization of this system can be obtained by using the formulae of Eqn.~\eqref{Box:PIs}. Then, the LPI from Lem.~\ref{stage1} can be solved using PIETOOLS. Finally, the LPI in Thm.~\ref{stage2} provides the SOF gain. A bisection algorithm can be used to maximize the decay rate $\alpha$ by solving Lem.~\ref{stage1} and Thm.~\ref{stage2} in each iteration. 
For $\alpha=1.8154$, the resulting controller is $u(t)=~\bmat{2374.12&321.31&-317.25&-209.37} y(t)$ and the maximum real part of the eigenvalues is estimated as $-2.2732$. The state trajectories of the resulting closed-loop system with initial condition $x(t)=\bmat{1&1&1&1}^T$, for all $t \in [-0.1, 0]$, are presented in Fig.~\ref{fig:1}.

By contrast, using the non-convex approach to SOF in~\cite{TDScontrol}, one obtains a gain of \[L~=~10^3\bmat{8.11 &4.99&-5.71&-2.48}\] and a maximum real part of eigenvalues estimated as $-1.4059$.
 \begin{figure}[t]
\centering
\begin{subfigure}{0.49\linewidth}
	\includegraphics[width=\linewidth, height=35mm]{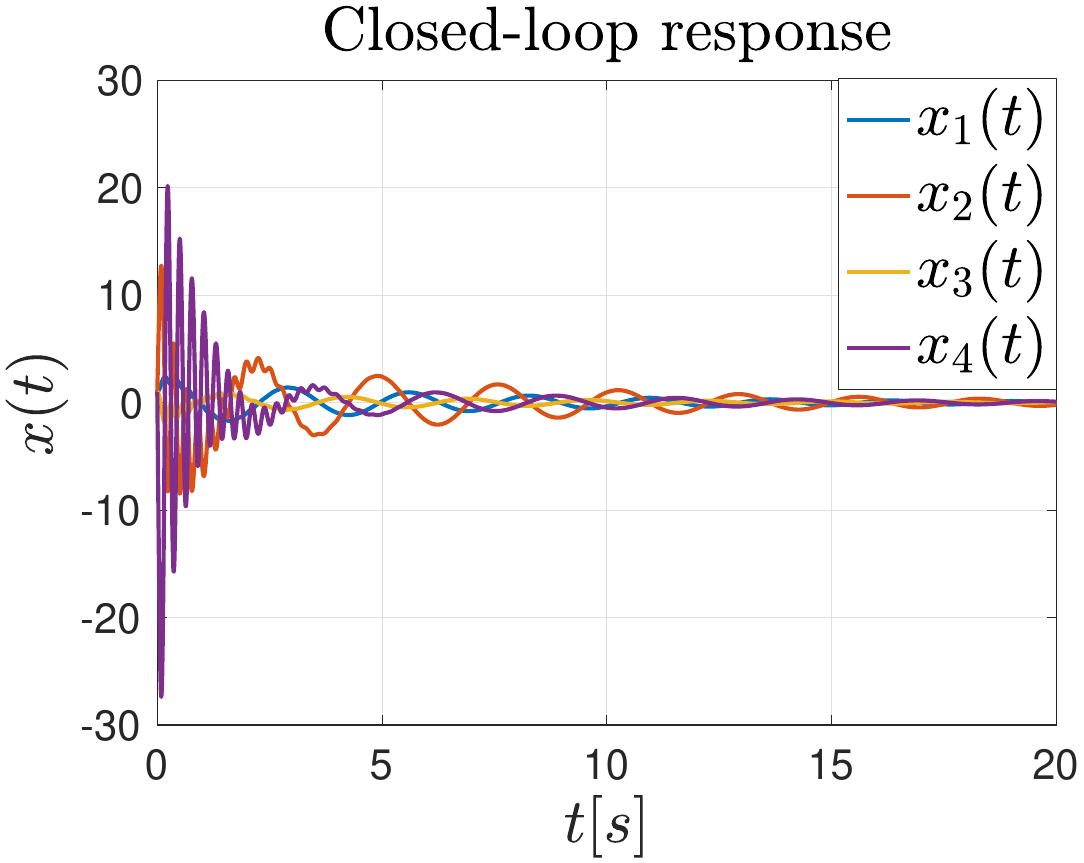}
	\caption{}
\end{subfigure}
 \hfill
\begin{subfigure}{0.49\linewidth}
\includegraphics[width=\linewidth,height=35mm]{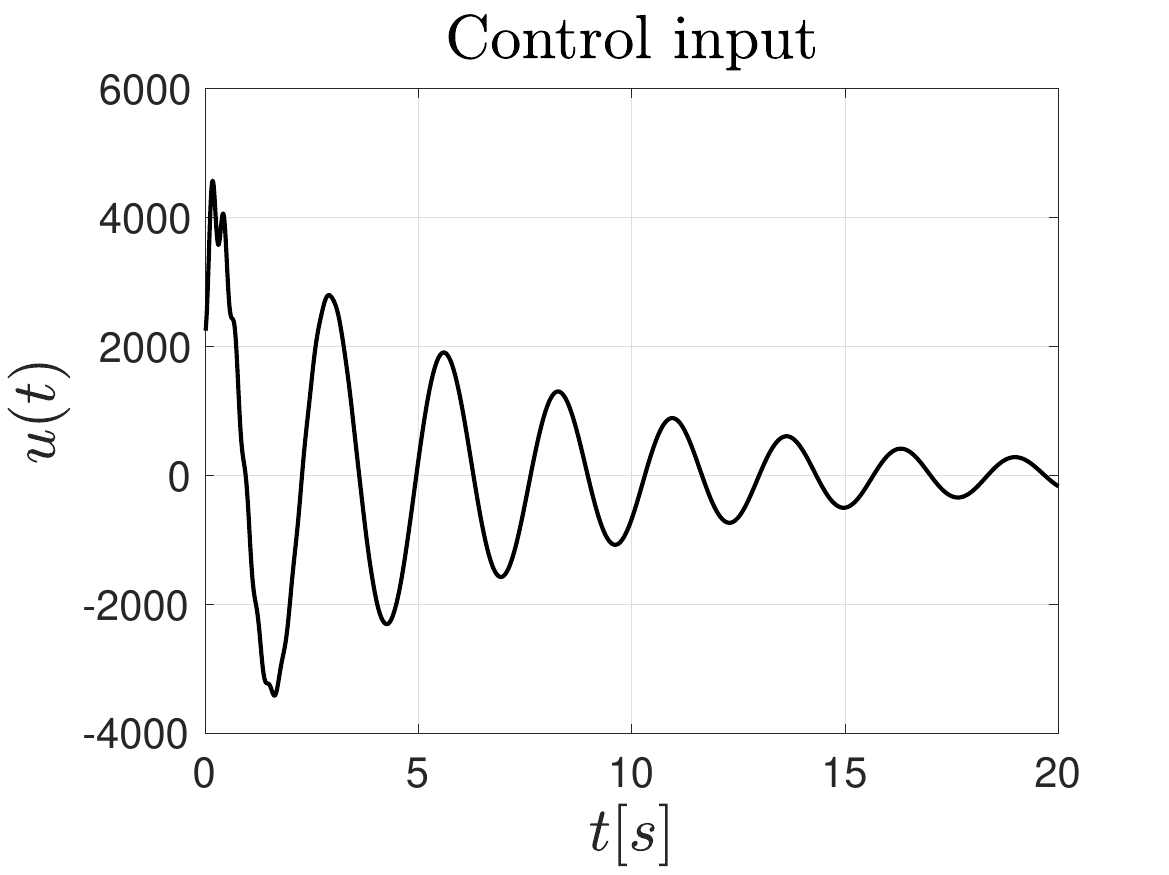}
	\caption{}
\end{subfigure}
\caption{(a) Trajectories of the states of the closed-loop system in Ex.~1, with $x(t)=~\bmat{1&1&1&1}^T$ for all $t \in [-0.1, 0]$. (b) The corresponding control input $u(t)$.}
\label{fig:1}
\end{figure}
\begin{figure}[t]
\centering
\begin{subfigure}{0.49\linewidth}
	\includegraphics[width=\linewidth, height=35mm]{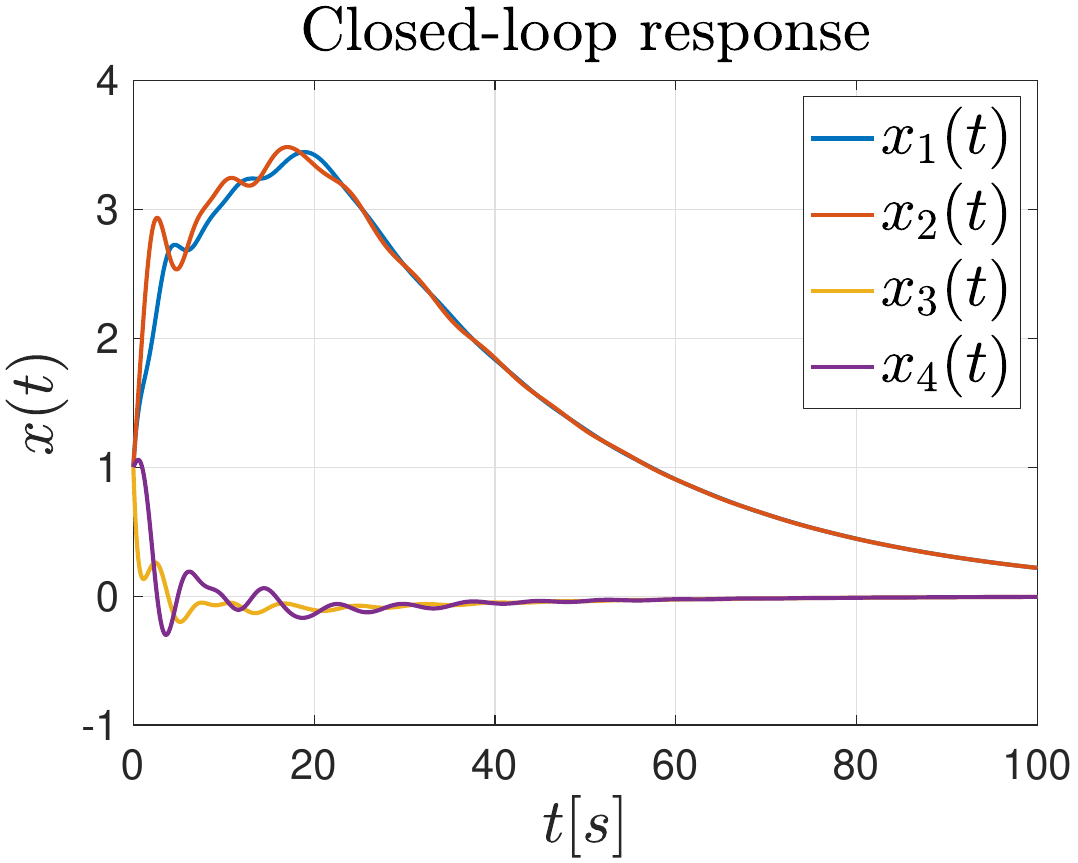}
	\caption{}
\end{subfigure}
\begin{subfigure}{0.49\linewidth}
\includegraphics[width=\linewidth,height=35mm]{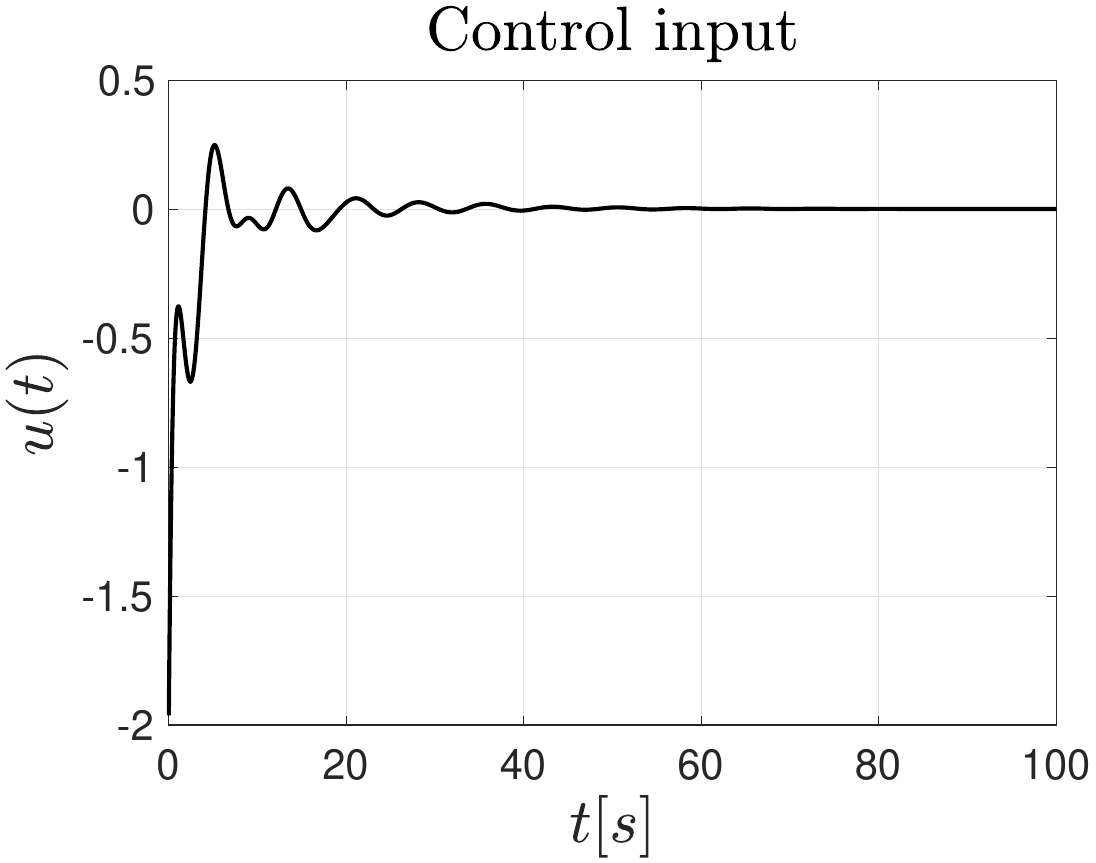}
	\caption{}
\end{subfigure}
\caption{(a) Trajectories of the states of the closed-loop system in Ex.~2, with $x(t)=~\bmat{1&1&1&1}^T$ for all $t \in [-20, 0]$. (b) The corresponding control input $u(t)$.}
\label{fig:Ex4}
\end{figure}
\begin{figure}[t]
\centering
\begin{subfigure}{0.49\linewidth}
	\includegraphics[width=\linewidth, height=35mm]{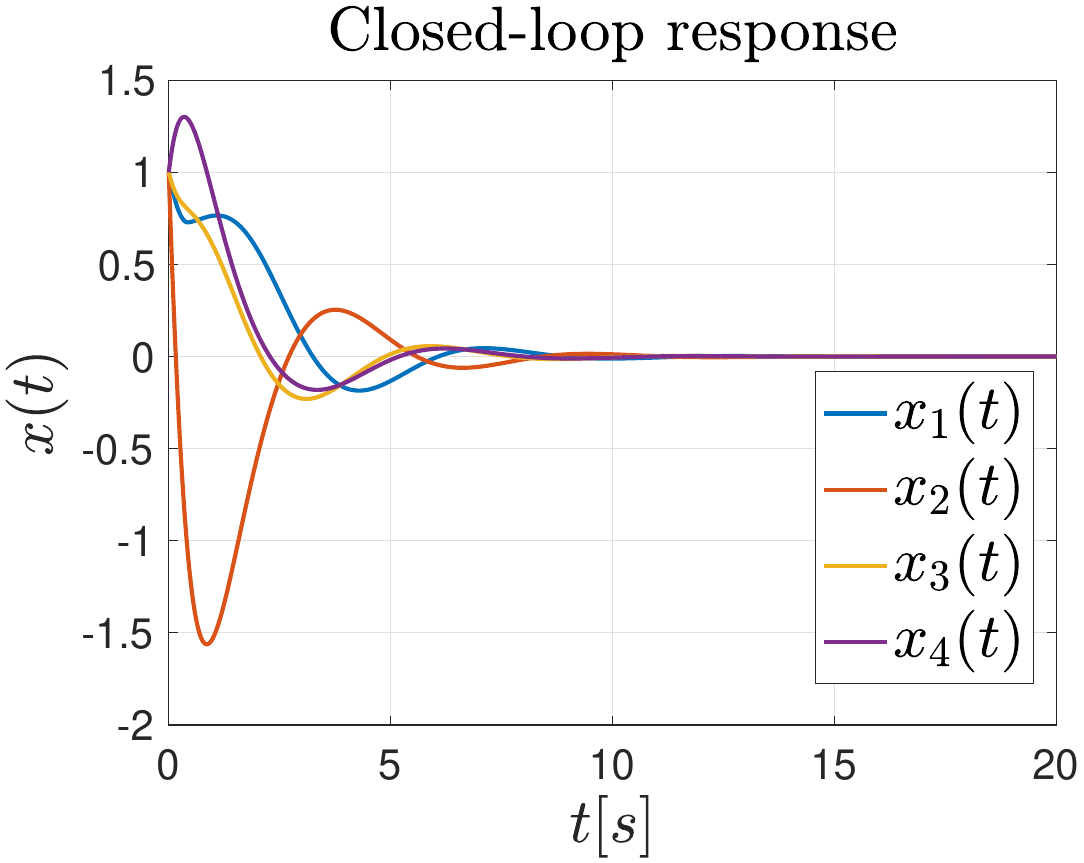}
	\caption{}
\end{subfigure}
\hfill
\begin{subfigure}{0.49\linewidth}
\includegraphics[width=\linewidth,height=35mm]{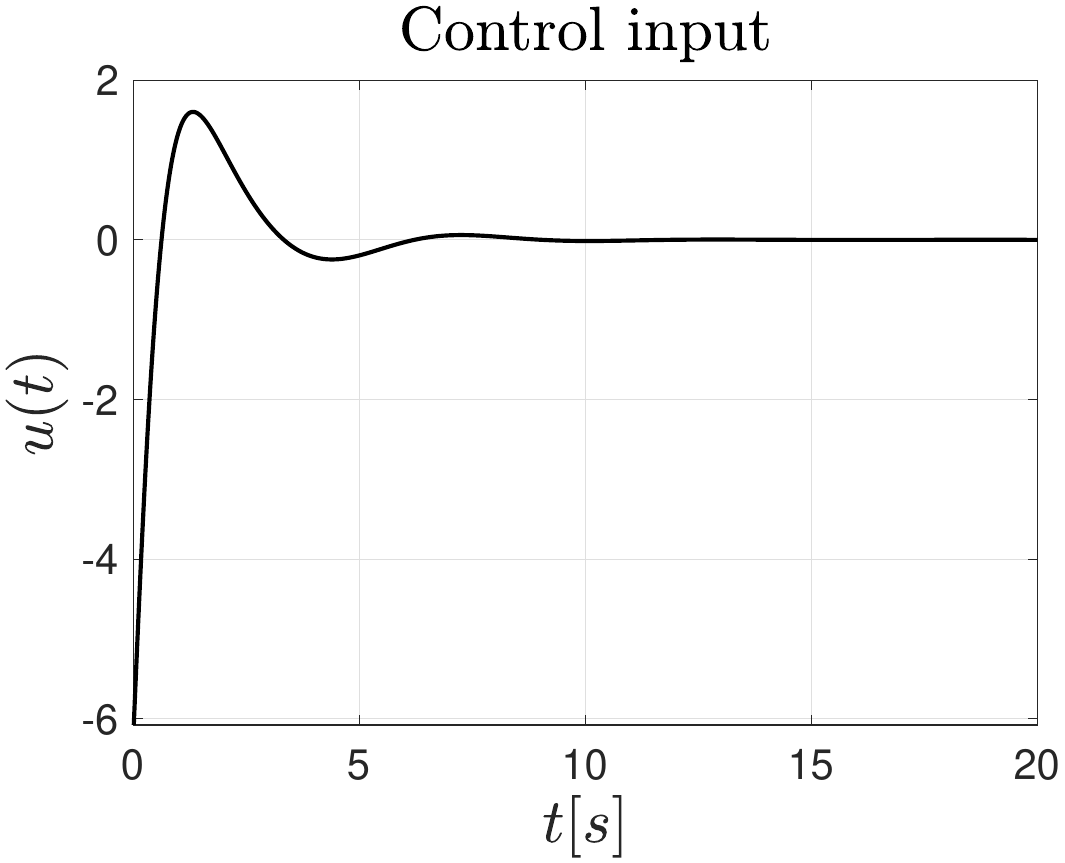}
	\caption{}
\end{subfigure}
\caption{(a) Trajectories of the states of the closed-loop system in Ex.~3, with $x(t)=~\bmat{1&1&1&1}^T$ for all $t \in [-0.45, 0]$. (b) The corresponding control input $u(t)$.}
\label{fig:Ex3}
\end{figure}
  \begin{figure}[t]
\centering
\begin{subfigure}{0.49\linewidth}
	\includegraphics[width=\linewidth, height=35mm]{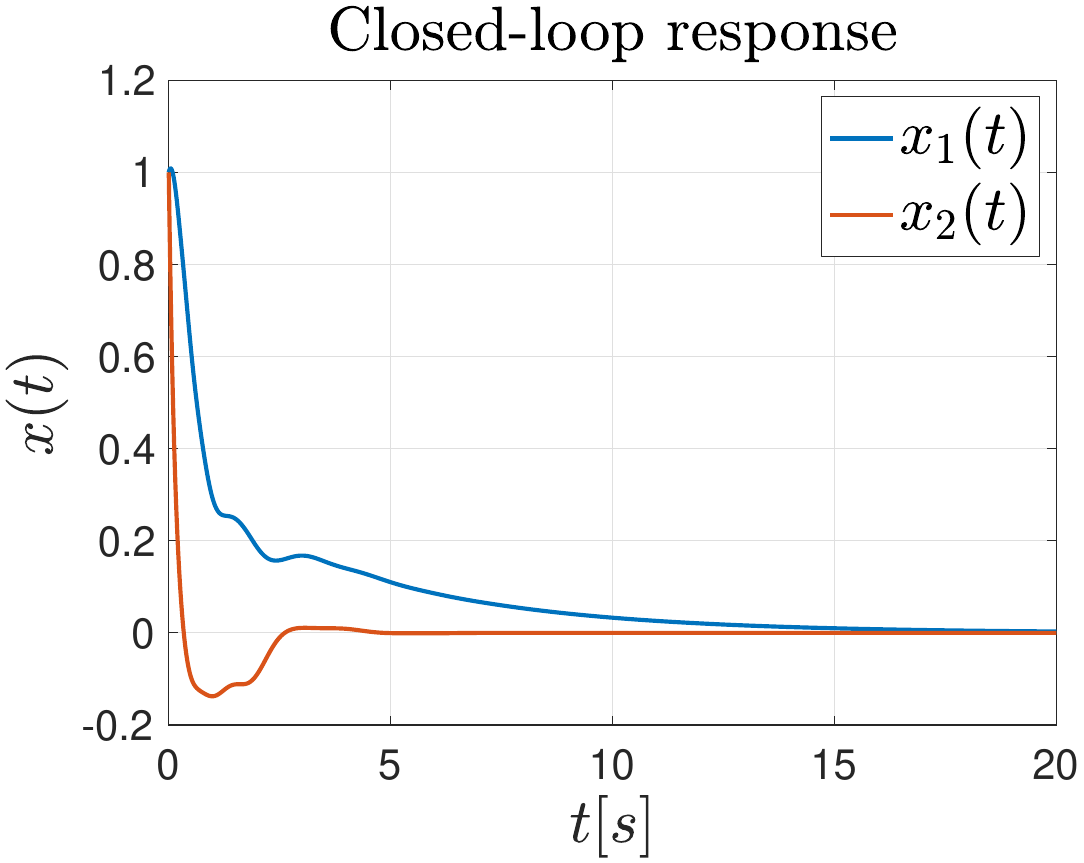}
	\caption{}
\end{subfigure}
\begin{subfigure}{0.49\linewidth}
	\includegraphics[width=\linewidth,height=35mm]{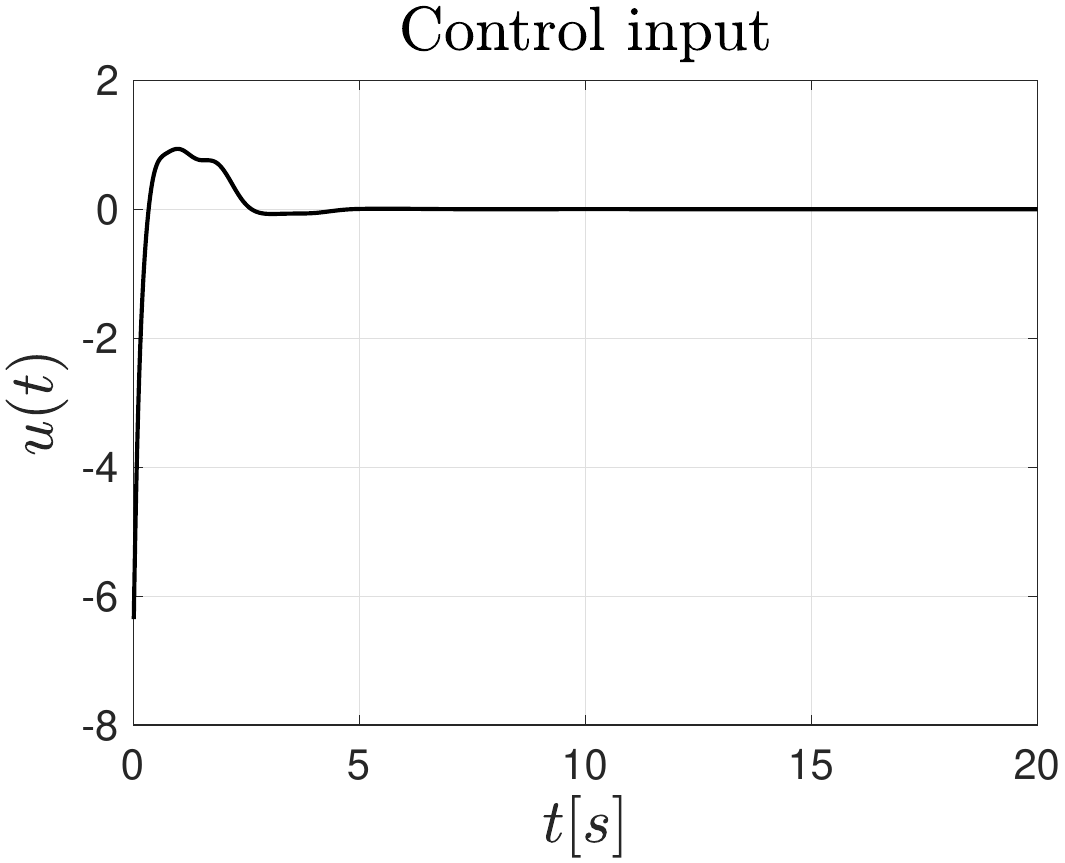}
	\caption{}
\end{subfigure}
\caption{(a) Trajectories of the states of the closed-loop system in Ex.~4, with $x(t)=\bmat{1&1}^T$ for all $t \in [-2, 0]$. (b) The corresponding control input $u(t)$.}
\label{fig:Ex5}
\end{figure}

\noindent \textbf{Example 2: }Consider the single-delay system from~\cite{hao2015static}, wherein 
\begin{align*}
A&=\bmat{0 & 0&1&0\\0 &0&0&1\\ -1 & 1&0&0\\1 &-1&0&0}, \, \quad A_1=0.1 \cdot A, \,\quad   B=\bmat{0\\0\\1\\0}, \notag\\
C&=\bmat{0 & I_2&0}, \,\qquad \qquad  C_1=0.
\end{align*}

 The LPI conditions of Lem.~\ref{stage1} and Thm.~\ref{stage2} were implemented with $\alpha = 10^{-12}$, and the SOF control law was obtained as $u(t)=\bmat{-0.055832&-1.9481} y(t)$ for $\tau=20 s$. The state trajectories with initial condition $x(t)=~\bmat{1&1&1&1}^T$, for all $t \in [-20, 0]$ are presented in Fig.~\ref{fig:Ex4}. 
 
   
   \noindent \textbf{Example 3: }Consider the single-delay system from~\cite{Fridman2009}, wherein
\begin{align*}
A&=\bmat{-0.2 & 0&0&0\\0 &0&0&-1\\ -1 & 0&-1&1\\0 &1&1&0}, \, A_1=\bmat{-0.8 & 0&1&0\\ 0 & 0&0&0\\0 &0&0&0\\0 &0&0&0}, \, \notag\\ B&=\bmat{0&1&0&0}^T,
C=\bmat{0 & I_2&0}, \, C_1=0.
\end{align*}
 The LPI conditions of Lem.~\ref{stage1} and Thm.~\ref{stage2} were implemented with $\alpha = 10^{-12}$, and the SOF control was obtained as $u(t)=\bmat{-2.8216&-3.392} y(t)$ for $\tau=0.45$. The state trajectories with initial condition $x(t)=\bmat{1&1&1&1}^T$, for all $t \in [-0.45, 0]$ are presented in Fig.~\ref{fig:Ex3}. It is possible to obtain feasible results for the corresponding system with a delay value $\tau$ up to $1.12 s$. Comparatively, applying the adapted LMIs of~\cite{hao2015static} to this system, a feasible solution cannot be obtained, highlighting the reduction in conservatism of the new procedure.

 \noindent \textbf{Example 4: }Consider the two-delay system from~\cite{peet2020convex}, wherein
\begin{align*}
A=&\bmat{-1 & 2\\ 0 &1}, \, A_1=\bmat{0.6 & -0.4\\ 0 & 0}, \, A_2=\bmat{0 & 0\\ 0 & -0.5} \notag\\ B&=\bmat{0&1}^T, \, C=\bmat{0 & 1} ,\, C_1=C_2=0. 
\end{align*}
 The LPI conditions of Lem.~\ref{stage1} and Thm.~\ref{stage2} were implemented with $\alpha = 10^{-12}$, and the SOF control law was obtained as $u(t)=-6.792y(t)$ for $\tau_{1}=1 s$ and $\tau_{2} = 2 s$. The state trajectories with initial condition $x(t)=\bmat{1&1}^T$, for all $t \in [-2, 0]$ are presented in Fig.~\ref{fig:Ex5}.

\section{Conclusion}

	In this paper, we have provided a two-step method for computing stabilizing SOF controllers for systems with multiple delays in the state and output. 
    Initially, the delay system is represented as a PIE, parameterized by elements of the PI operator algebra. In the first step, a state-feedback controller is synthesized by solving an LPI optimization problem. In the second step, an extension of the projection lemma to PI operators allows for the SOF closed-loop stability conditions to be expressed as a PI inequality 
    in terms of 
    both the known state-feedback gains as well as the unknown output feedback gains. A variable substitution then results in a convex LPI, which can be tested using software such as PIETOOLS. 
     Success of this method requires the existence of a quadratic Lyapunov function, which can be used to certify stability of the SOF controller as well as a state-feedback controller computed in the first step. As a result, the accuracy of the approach could potentially depend on the gains chosen in the first step. However, numerical validation indicates the approach is highly effective regardless of how the gains are chosen in the first step.


\bibliographystyle{ieeetr}
\bibliography{Braghini_CDC2025_SOF}             

@string{TAC="IEEE Transactions on Automatic Control"}

@string{AUT="Automatica"}

@string{CDC="Proceedings of the IEEE Conference on Decision and Control"}

@string{ACC="Proceedings of the American Control Conference"}

@article{parlakci2024robust,
  title={Robust static output feedback {${H}_{\infty}$} controller design for linear parameter-varying time delay systems},
  author={Parlakci, Mehmet Nur Alpaslan},
  journal={Circuits, Systems, and Signal Processing},
  volume={43},
  number={2},
  pages={843--864},
  year={2024},
  publisher={Springer}
}

@inproceedings{toker1995np,
  title={On the {NP}-hardness of solving bilinear matrix inequalities and simultaneous stabilization with static output feedback},
  author={Toker, Onur and Ozbay, Hitay},
  booktitle=ACC,
  volume={4},
  pages={2525--2526},
  year={1995},
  organization={IEEE}
}

@article{TDScontrol,
title = {{TDS-CONTROL}: a {MATLAB} package for the analysis and controller-design of time-delay systems*},
journal = {IFAC-PapersOnLine},
volume = {55},
number = {16},
pages = {272-277},
year = {2022},
doi = {https://doi.org/10.1016/j.ifacol.2022.09.036},
author = {Pieter Appeltans and Haik Silm and Wim Michiels}

}

@article{peet2021representation,
  title={Representation of networks and systems with delay: {DDEs}, {DDFs}, {ODE--PDEs} and {PIEs}},
  author={Peet, M.},
  journal=AUT,
  volume={127},
    note = {{A}rt. no 109508},
  year={2021},
  publisher={Elsevier}
}

@article{Hajdu2020,
title = {Robust stability of milling operations based on pseudospectral approach},
journal = {International Journal of Machine Tools and Manufacture},
volume = {149},
note = {{A}rt. no 103516},
year = {2020},
author = {David Hajdu and Francesco Borgioli and Wim Michiels and Tamas Insperger and Gabor Stepan}
}

@inproceedings{peaucelle2001efficient,
  title={An efficient numerical solution for {${H}_{2}$} static output feedback synthesis},
  author={Peaucelle, D. and Arzelier, D.},
  booktitle={2001 European control conference (ECC)},
  pages={3800--3805},
  year={2001},
  organization={IEEE},
  location={Porto, Portugal}
}

@ARTICLE{Fridman2009,
  author={Han, X. R. and Fridman, Emilia and Spurgeon, Sarah K. and Edwards, Chris},
  journal={IEEE Transactions on Industrial Electronics}, 
  title={On the Design of Sliding-Mode Static-Output-Feedback Controllers for Systems With State Delay}, 
  year={2009},
  volume={56},
  number={9},
  pages={3656-3664},
  doi={10.1109/TIE.2009.2023635}}

@article{shivakumar2022dual,
  title={Dual Representations and {${H}_{\infty}$}-Optimal Control of Partial Differential Equations},
  author={Shivakumar, S. and Das, A. and Weiland, S. and Peet, M.},
number={2208.13104},
  journal={arXiv},
  year={2022}
}

@inproceedings{sandou2008receding,
  title={Receding horizon climate control in metal mine extraction rooms},
  author={Sandou, Guillaume and Witrant, Emmanuel and Olaru, Sorin and Niculescu, Silviu-Iulian},
  booktitle={IEEE International Conference on Automation Science and Engineering},
  pages={839--844},
  year={2008},
  organization={IEEE},
  location={Arlington, VA, USA}
}

@article{liu2015static,
  title={Static output feedback stabilization for systems with time-varying delay based on a matrix transformation method},
  author={Liu, ZhenWei and Zhang, HuaGuang and Sun, QiuYe},
  journal={Science China. Information Sciences},
  volume={58},
  number={1},
  pages={1--13},
  year={2015},
  publisher={Springer Nature BV}
}

@article{jagt2024h,
  title={{$H_{\infty}$}-Optimal Estimator Synthesis for Coupled Linear {2D PDEs }using Convex Optimization},
  author={Jagt, Declan S and Peet, Matthew M.},
number={2402.05061},
  journal={arXiv},
  year={2024}
}

@article{peet2020convex,
  title={A Convex Solution of the ${H}_{\infty}$-Optimal Controller Synthesis Problem for Multidelay Systems},
  author={Peet, Matthew M},
  journal={SIAM Journal on Control and Optimization},
  volume={58},
  number={3},
  pages={1547--1578},
  year={2020},
  publisher={SIAM}
}

@article{hao2015static,
  title={Static output-feedback controller synthesis with restricted frequency domain specifications for time-delay systems},
  author={Hao, Yuqing and Duan, Zhisheng},
  journal={IET Control Theory \& Applications},
  volume={9},
  number={10},
  pages={1608--1614},
  year={2015},
  publisher={Wiley Online Library}
}

@article{peet2021partial,
  title={A Partial Integral Equation {(PIE)} representation of coupled linear {PDEs} and scalable stability analysis using {LMIs}},
  author={Peet, M.},
  journal=AUT,
  volume={125},
  note = {{A}rt. no 109473},
  year={2021},
  publisher={Elsevier}
}

@article{manual2024,
  title={{PIETOOLS} 2024: {User Manual}},
  author={Shivakumar, S. and Jagt, D. and Braghini, D. and Das, A. and Peet, M.},
  journal={arXiv},
  number={2101.02050},
  year={2021},
    url   = {https://control.asu.edu/pietools/pietools},
}

@article{peet2020piesim,
  title={A New Treatment of Boundary Conditions in {PDE} Solution with {Galerkin} Methods via {Partial Integral Equation} Framework},
  author={Peet, Y. and Peet, M.},
  journal={Journal of Computational and Applied Mathematics},
 volume={442},
  note = {{A}rt. no 115673},
  year={2024},
}

@article{danilo,
  title={Computing Optimal Upper Bounds on the ${H}_2$-norm of {ODE-PDE} Systems using {Linear Partial Inequalities}},
  author={Braghini, D. and Peet, M.},
  journal={IFAC-PapersOnLine},
  volume={56},
  number={2},
  pages={6994--6999},
  year={2023},
  publisher={Elsevier}
}

@article{shivakumar2024extension,
  title={Extension of the {P}artial {I}ntegral {E}quation representation to {GPDE} input-output systems},
  author={Shivakumar, S. and Das, A. and Weiland, S. and Peet, M.},
  journal=TAC,
  year={2024},
  publisher={IEEE}
}

@INPROCEEDINGS{das_2019CDC,
  author =       {A. Das and S. Shivakumar and S. Weiland and M. Peet},
  title =        {{${H}_{\infty}$}-Optimal Estimation for Linear Coupled {PDE} Systems},
  booktitle =    {2019 IEEE 58th Conference on Decision and Control (CDC)},
  year =         {2019},
  pages={262-267},
}

@article{pipeleers2009extended,
  title={Extended {LMI} characterizations for stability and performance of linear systems},
  author={Pipeleers, Goele and Demeulenaere, Bram and Swevers, Jan and Vandenberghe, Lieven},
  journal={Systems \& Control Letters},
  volume={58},
  number={7},
  pages={510--518},
  year={2009},
  publisher={Elsevier}
}

@article{huynh2019static,
  title={Static output feedback control of positive linear systems with output time delays},
  author={Huynh, Van Thanh and Nguyen, Cuong M and Trinh, Hieu},
  journal={International Journal of Systems Science},
  volume={50},
  number={15},
  pages={2815--2823},
  year={2019},
  publisher={Taylor \& Francis}
}
\end{document}